\documentclass[11pt,letterpaper]{article}

\usepackage{amsmath,amssymb,amsfonts,mathrsfs} 
\usepackage{graphicx}
\usepackage[usenames,dvipsnames]{color}
\usepackage[normalem]{ulem}

\usepackage[margin=1in]{geometry}
\newenvironment{proof}{\noindent{\it Proof.\,}}{\hfill$\Box$}

\newtheorem{thm}{Theorem}
\newtheorem{lem}[thm]{Lemma}

\newtheorem{defn}[thm]{Definition}

\newcommand{\real}{\ensuremath {\mathbb R} }
\newcommand{\ent}{\ensuremath {\mathbb Z} }
\newcommand{\nat}{\ensuremath {\mathbb N} }

\newcommand{\remove}[1] {}

\newcommand{\ex} {{\bf E}}
\newcommand{\pr} {{\bf Pr}}

\newcommand{\eps}{\varepsilon}
\renewcommand{\phi}{\varphi}

\newcommand{\A}{{\cal{A}}}
\newcommand{\G}{{\mathcal{G}}}

\DeclareMathOperator{\Bin}{Bin}
\DeclareMathOperator{\vol}{vol}

\def\pr{{\mathbb P}}

\title{Modularity of complex networks models}
\author{
Liudmila Ostroumova Prokhorenkova$^{1,2}$
\and
Pawe{\l} Pra{\l}at$^{3,4}$
\and
Andrei Raigorodskii$^{1,2,5,6}$
}
\date{%
    $^1$Moscow Institute of Physics and Technology, Moscow, Russia\\%
    $^2$Yandex, Moscow, Russia\\%
	$^3$Ryerson University, Toronto, ON, Canada\\
    $^4$The Fields Institute for Research in Mathematical Sciences, Toronto, ON, Canada\\
	$^5$Moscow State University, Moscow, Russia\\
    $^6$Buryat State Unversity, Ulan-Ude, Buryat Republic, Russia
}

\begin{document}

\maketitle

\begin{abstract}
Modularity is designed to measure the strength of division of a network into clusters (known also as communities). Networks with high modularity have dense connections between the vertices within clusters but sparse connections between vertices of different clusters. As a result, modularity is often used in optimization methods for detecting community structure in networks, and so it is an important graph parameter from a practical point of view. Unfortunately, many existing non-spatial models of complex networks do not generate graphs with high modularity; on the other hand, spatial models naturally create clusters. We investigate this phenomenon by considering a few examples from both sub-classes. We prove precise theoretical results for the classical model of random $d$-regular graphs as well as the preferential attachment model, and contrast these results with the ones for the spatial preferential attachment (SPA) model that is a model for complex networks in which vertices are embedded in a metric space, and each vertex has a sphere of influence whose size increases if the vertex gains an in-link, and otherwise decreases with time.
The results obtained in this paper can be used for developing statistical tests for models selection and to measure statistical significance of clusters observed in complex networks.
\end{abstract}

\section{Introduction}

Many social, biological, and information systems can be represented by networks, whose vertices are items and links are relations between these items~\cite{BA_Review,BioInfoPrior,Networks,Chayes}. That is why the evolution of complex networks attracted a lot of attention in recent years and there has been a great deal of interest in modelling of these networks~\cite{Math_Results,Costa,Newman3}. The hyperlinked structure of the Web, citation patterns, friendship relationships, infectious disease spread are seemingly disparate linked data sets which have fundamentally very similar natures. Indeed, it turns out that many real-world networks have some typical properties: heavy tailed degree distribution, small diameter, high clustering coefficient, and others~\cite{Newman1,Newman2,Watts}. Such properties are well-studied both in real-world networks and in many theoretical models.

Another important property of complex networks is their community structure, that is, the organization of vertices in clusters, with many edges joining vertices of the same cluster and comparatively few edges joining vertices of different clusters~\cite{Fortunato,Girvan}. In social networks communities may represent groups by interest, in citation networks they correspond to related papers, in the Web communities are formed by pages on related topics, etc. Being able to identify communities in a network could help us to exploit this network more effectively. For example, clusters in citation graphs may help to find similar scientific papers, discovering users with similar interests is important for targeted advertisement, clustering can also be used for network compression and visualization.

The key ingredient for many clustering algorithms is \textit{modularity}, which is at the same time a global criterion to define communities, a quality function of community detection algorithms, and a way to measure the presence of community structure in a network. Modularity was introduced by Newman and Girvan~\cite{Mod2} and it is based on the comparison between the actual density of edges inside a community and the density one would expect to have if the vertices of the graph were attached at random, regardless of community structure. 

Unfortunately, modularity is not a well studied parameter for the existing random graph models, at least from a rigorous, theoretical point of view. We are only aware about results for binomial random graphs $G(n,p)$ and random $d$-regular graphs (see Section~\ref{sec:previous_results} for more details). In this paper, we continue investigating random $d$-regular graphs and obtain new upper bounds for their modularity.  Then we move to the \emph{preferential attachment model}, introduced by Barab\'asi and Albert~\cite{BA}, which is probably the most well-studied model of complex networks. For this model no results on modularity are known and we obtain both lower and upper bounds. In fact, one of the lower bound we present holds for all graphs with average degree $d$ and sublinear maximum degree.

As expected, the models discussed above, as well as many others, have a common weakness of low modularity. One family of models which overcomes this deficiency is the family of spatial (or geometric) models, wherein the vertices are embedded in a metric space such that similar vertices are closer to each other than dissimilar ones. The underlying geometry of spatial models naturally leads to the emergence of clusters. We prove this statement rigorously for one example of a geometric model, the Spatial Preferential Attachment model introduced in~\cite{spa1}.

This paper is a journal version of~\cite{proceedings} and is structured as follows. In the next section, we formally define modularity, discuss several random graph models and present known results on modularity in these models. In Sections~\ref{sec:d-reg}, \ref{sec:PA} and~\ref{sec:SPA} we analyze modularity in random $d$-regular graphs, preferential attachment and SPA models, respectively. In Section~\ref{sec:lower} we discuss lower bounds for modularity of forests and constant average degree graphs. Section~\ref{sec:conclusion} concludes the paper and outlines the directions for future research.

\section{Preliminaries}

\subsection{Modularity}\label{sub:modularity}

The definition of modularity was first introduced by Newman and Girvan in~\cite{Mod2}. Since then, many popular and applied algorithms used to find clusters in large data-sets are based on finding partitions with high modularity~\cite{Clauset,Mod1,Newman}. The modularity function favours partitions in which a large proportion of the edges fall entirely within the parts and biases against having too few or too unequally sized parts. Formally, for a given partition $\A = \{A_1, \ldots, A_k\}$ of the vertex set $V(G)$, let
\begin{equation}\label{eq:q_A}
q_{\A} = \sum_{A \in \A} \left( \frac {e(A)}{|E(G)|} - \frac { (\sum_{v \in A} \deg(v))^2 }{4|E(G)|^2} \right),
\end{equation}
where $e(A) = |\{ uv \in E(G) : u,v \in A\}|$ is the number of edges in the graph induced by the set $A$. The first term, $\sum_{A \in \A} \frac {e(A)}{|E(G)|}$, is called the \emph{edge contribution}, whereas the second one, $\sum_{A \in \A} \frac { (\sum_{v \in A} \deg(v))^2 }{4|E(G)|^2}$, is called the \emph{degree tax}. 
It is easy to see that $q_{\A}$ is always smaller than one. Also, if $\A = \{V(G)\}$, then $q_{\A} = 0$.

The \emph{modularity} $q^*(G)$ is defined as the maximum of $q_{\A}$ over all possible partitions $\A$ of $V(G)$; that is,
$$
q^*(G) = \max_{\A} q_{\A} (G).
$$
In order to maximize $q_{\A} (G)$ one wants to find a partition with large edge contribution subject to small degree tax. If $q^*(G)$ approaches 1 (which is the maximum), we observe a strong community structure; conversely, if $q^*(G)$ is close to zero, we are given a graph with no community structure. 

Modularity is known to have some weaknesses, as discussed in~\cite{Fortunato}. For example, \cite{ResolutionLimit} shows that this measure fails to detect communities if their sizes are too small. However, despite this, modularity still remains to be the most popular measure used by many well known clustering algorithms~\cite{Clauset,Mod1,Newman}.

\subsection{Random graph models}

\paragraph{Random $d$-regular graphs.}


We consider
the probability space of \emph{random $d$-regular graphs} with uniform probability distribution. This space is denoted $\mathcal{G}_{n,d}$, and asymptotics are for $n\to\infty$ with $d\ge 2$ fixed, and $n$ even if $d$ is odd.

We say that an event in a probability space holds \emph{asymptotically almost surely} (or \emph{a.a.s.}) if the probability that it holds tends to $1$ as $n$ goes to infinity. Since we aim for results that hold a.a.s., we will always assume that $n$ is large enough.

\paragraph{Preferential Attachment.}

The \emph{Preferential Attachment} (PA) model~\cite{BA} was an early stochastic model of complex networks. We will use the following precise definition of the model, as considered by Bollob\'as and Riordan in~\cite{BR} as well as Bollob\'as, Riordan, Spencer, and Tusn\'ady~\cite{BRST}.

Let $G_1^0$ be the null graph with no vertices (or let $G_1^1$ be the graph with one vertex, $v_1$, and one loop). The random graph process $(G_1^t)_{t \ge 0}$ is defined inductively as follows. Given $G_1^{t-1}$, we form $G_1^t$ by adding a vertex $v_t$ together with a single edge between $v_t$ and $v_i$, where $i$ is selected randomly with the following probability distribution:
$$
\pr(i = s) =
\begin{cases}
\deg(v_s,t-1) / (2t-1) & 1 \le s \le t-1, \\
1/(2t-1) & s=t,
\end{cases}
$$
where  $\deg(v_s,t-1)$ denotes the degree of $v_s$ in $G_1^{t-1}$ (loops are counted twice). In other words, at $t$-th step of the process we send an edge $e$ from $v_t$ to a random vertex $v_i$, where the probability that a vertex is chosen is proportional to its current degree, counting $e$ as already contributing one to the degree of $v_t$.

For $m \in \nat \setminus \{1\}$, the process $(G_m^t)_{t \ge 0}$ is defined similarly with the only difference that $m$ edges are added to $G_m^{t-1}$ to form $G_m^t$ (one at a time), counting previous edges as already contributing to the degree distribution. Equivalently, one can define the process $(G_m^t)_{t \ge 0}$ by considering the process $(G_1^t)_{t \ge 0}$ on a sequence $v'_1, v'_2, \ldots$ of vertices; the graph $G_m^t$ is formed from $G_1^{tm}$ by identifying vertices $v'_1, v'_2, \ldots, v'_m$ to form $v_1$, identifying vertices $v'_{m+1}, v'_{m+2}, \ldots, v'_{2m}$ to form $v_2$, and so on. Note that in this model $G_m^t$ is in general a multigraph, possibly with multiple edges between two vertices (if $m\ge2$) and self-loops. 

It was shown in~\cite{BRST} that for any $m \in \nat$ a.a.s.\ the degree distribution of $G_m^n$ follows a power law: the number of vertices with degree at least $k$ falls off as $(1+o(1)) ck^{-2} n$  for some explicit constant $c=c(m)$ and large $k \le n^{1/15}$. Also, in the case $m=1$, each vertex sends an edge either to itself or to an earlier vertex, so $G_1^n$ is a forest with each component containing a single looped vertex. The expected number of components is then $\sum_{t=1}^n 1/(2t-1) \sim (1/2) \log n$ and, since events are independent, we derive that a.a.s.\ there are $(1/2+o(1)) \log n$ components in $G_1^n$ by Chernoff's bound. 
In contrast, for the case $m \ge 2$ it is known that a.a.s.\ $G_m^n$ is connected and its diameter is $(1+o(1)) \log n / \log \log n$~\cite{BR}. 

\paragraph{Spatial Preferential Attachment.}

The \emph{Spatial Preferential Attachment} (SPA) model~\cite{spa1}, designed as a model for the World Wide Web, combines geometry and preferential attachment, as its name suggests. Setting the SPA model apart is the incorporation of `spheres of influence' to accomplish preferential attachment: the greater the degree of a vertex, the larger its sphere of influence, and hence the higher the likelihood of the vertex gaining more neighbours. 

We now give a precise description of the SPA model. Let $S = [0,1]^m$ be the unit hypercube in $\real^m$, equipped with the torus metric derived from any of the $L_p$ norms. This means that for any two points $x$ and $y$ in $S$,
\[
d(x,y)=\min \big\{ ||x-y+u||_p\,:\,u\in \{-1,0,1\}^m \big\}.
\]
The torus metric thus `wraps around' the boundaries of the  unit square; this metric was chosen to eliminate boundary effects.
The parameters of the model consist of the \emph{link probability} $p\in[0,1]$, and two positive constants $A_1$ and $A_2$, which, in order to avoid the resulting graph becoming too dense, must be chosen so that $pA_1 < 1$.  The SPA model generates stochastic sequences of directed graphs $(G_{t}:t\geq 0)$, where $G_{t}=(V_{t},E_{t})$, and $V_{t}\subseteq S$. Let $\deg^{-}(v,t)$ be the in-degree of the vertex $v$ in $ G_{t}$, and $\deg^+(v,t)$ its out-degree. We define the \emph{sphere of influence} $S(v,t)$ of the vertex $v$ at time $t\geq 1$ to be the ball centered at $v$ with volume $|S(v,t)|$ defined as follows:
\begin{equation}\label{eq:volume}
|S(v,t)|=\min\left\{\frac{A_1{\deg}^{-}(v,t)+A_2}{t},1\right\}.
\end{equation}

The process begins at $t=0$, with $G_0$ being the null graph. Time step $t$, $t\geq 1$, is defined to be the transition between $G_{t-1}$ and $G_t$. At the beginning of each time step $t$, a new vertex $v_t$ is chosen \emph{uniformly at random} from $S$, and added to $V_{t-1}$ to create $V_{t}$. Next, independently, for each vertex $u\in V_{t-1}$ such that $v_t \in S(u,t-1)$, a directed link $(v_{t},u)$ is created with probability $p$. Thus, the probability that a link $(v_t,u)$ is added in time-step $t$ equals $p\,|S(u,t-1)|$.

The SPA model produces scale-free networks, which exhibit many of the characteristics of real-life networks (see~\cite{spa1,spa2}). In~\cite{spa3}, it was shown that the SPA model gave the best fit, in terms of graph structure, for a series of social networks derived from Facebook. In~\cite{spa4}, some properties of common neighbors were used to explore the underlying geometry of the SPA model and quantify vertex similarity based on distance in the space.  However, the distribution of vertices in space was assumed to be uniform~\cite{spa4} and so in~\cite{spa5} non-uniform distributions were investigated which is clearly a more realistic setting.

\subsection{Previous results on modularity}\label{sec:previous_results}

In this section we discuss known bounds for modularity in different random graph models.

The \emph{isoperimetric number} (known also as edge expansion) of a graph $G$ is defined as
$$
\rho(G) = \min_{V(G) = V_1 \cup V_2} \frac {e(V_1,V_2)} { \min \{|V_1|,|V_2|\}},
$$
where $e(V_1,V_2) = |\{ uv \in E(G) : u \in V_1, v \in V_2\}|$ is the number of edges between the sets $V_1$ and $V_2$. The following result was shown by McDiarmid and Skerman in~\cite{Colin1}. Let $G$ be any $d$-regular graph on $n$ vertices. Then, the following useful upper bound on the modularity is almost immediate:
\begin{equation}\label{eq:upper_trivial}
q^*(G) \le \max\{1-\rho(G)/d, 3/4\}.
\end{equation}
Turning to random $d$-regular graphs, Bollob\'as in~\cite{bollobas1} showed that a.a.s.\ $\rho(\G_{n,d}) \ge (1-\eta)d/2$, where $0 < \eta < 1$ is such that $2^{4/d} < (1-\eta)^{1-\eta} (1+\eta)^{1+\eta}$ and so a.a.s.
$$
q^*(\G_{n,d}) \le U_1 = U_1(d) := \max\{1/2+\eta/2, 3/4\}.
$$
As a result, we get the first non-trivial upper bounds for $q^*(\G_{n,d})$ presented in Table~\ref{tab:upper_bound} that hold a.a.s. 

In~\cite{Colin1}, the bound~(\ref{eq:upper_trivial}) was slightly improved when the maximum size of parts in our partition is restricted. Formally, given $\delta > 0$, for a graph $G$ with $n \ge 1/\delta$ vertices, they define $q_{\delta}(G)$ to be the maximum modularity of all partitions for $G$ such that each part has size at most $\delta n$. They show that for any $\eps > 0$ there exists $\delta > 0$ such any $d$-regular graph with at least $1/\delta$ vertices satisfies
$$
q_{\delta}(G) \le 1-2\rho(G)/d + \eps.
$$
Again, using the result of Bollob\'as we get that there exists $\delta > 0$ such that
$$
U_2 = U_2(d) := \eta + \eps
$$
serves as an upper bound that holds a.a.s.\ for $q_{\delta}(\G_{n,d})$; again, see Table~\ref{tab:upper_bound} for numerical values for small values of $d$. It is straightforward to see that $(G) \ge d/2 - \sqrt{ (\log 2) d }$ (see, for example,~\cite{bollobas1}) and so, in particular, $U_2$ can be made arbitrarily small by taking $d$ large enough (and $\delta$ small enough). However, let us note that these upper bounds for $q_{\delta}(\G_{n,d})$, while useful, cannot be directly translated into any bound for $q^*(\G_{n,d})$.

\begin{table}
\begin{center}
\caption{Upper bounds $U_1$, $U_3$ for $q^*(\G_{n,d})$ and $U_2$ for $q_{\delta}(\G_{n,d})$}
  \begin{tabular}{ | c || c | c | c |}
    \hline
    $d$ & $U_1$ & $U_2$ & $U_3$ \\ \hline \hline
    3 & 0.9386 & 0.8771 & 0.8038 \\
    4 & 0.8900 & 0.7800 & 0.6834 \\
    5 & 0.8539 & 0.7078 & 0.6024 \\
    6 & 0.8261 & 0.6521 & 0.5435 \\
    7 & 0.8038 & 0.6076 & 0.4984\\
    8 & 0.7855 & 0.5710 & 0.4624 \\
    9 & 0.7702 & 0.5403 & 0.4330 \\
    10 & 0.7570 & 0.5140 & 0.4083 \\
    \hline
  \end{tabular}
\label{tab:upper_bound} 
\end{center}
\end{table}

Investigating random $d$-regular graphs continues in~\cite{Colin2}, a very recent paper. In fact, the numerical upper bound presented in Section~\ref{sec:numerical}, as well as the result in Theorem~\ref{thm:U3_explicit}, are obtained independently there. Moreover,~\cite{Colin2} investigates the class of graphs whose product of treewidth and maximum degree is much less than the number of edges. Their result shows, for example, that random planar graphs typically have modularity close to 1, which is another indication that clusters naturally emerge where geometry is included. Also, a particular case of their theorem shows that trees with maximum degree $o(n)$ have asymptotic modularity one.

\section{Random $d$-regular graphs}\label{sec:d-reg}

\subsection{Pairing model}

Instead of working directly in the uniform probability space of random regular graphs on $n$ vertices $\mathcal{G}_{n,d}$, we use the \textit{pairing model} (also known as the \textit{configuration model}) of random regular graphs, first introduced by Bollob\'{a}s~\cite{bollobas2}, which is described next. Suppose that $dn$ is even, as in the case of random regular graphs, and consider $dn$ points partitioned into $n$ labelled buckets $v_1,v_2,\ldots,v_n$ of $d$ points each. A \textit{pairing} of these points is a perfect matching into $dn/2$ pairs. Given a pairing $P$, we may construct a multigraph $G(P)$, with loops allowed, as follows: the vertices are the buckets $v_1,v_2,\ldots, v_n$, and a pair $\{x,y\}$ in $P$ corresponds to an edge $v_iv_j$ in $G(P)$ if $x$ and $y$ are contained in the buckets $v_i$ and $v_j$, respectively. It is an easy fact that the probability of a random pairing corresponding to a given simple graph $G$ is independent of the graph, hence the restriction of the probability space of random pairings to simple graphs is precisely $\mathcal{G}_{n,d}$. Moreover, it is well known that a random pairing generates a simple graph with probability asymptotic to $e^{-(d^2-1)/4}$ depending on $d$, so that any event holding a.a.s.\ over the probability space of random pairings also holds a.a.s.\ over the corresponding space $\mathcal{G}_{n,d}$. For this reason, asymptotic results over random pairings suffice for our purposes. For more information on this model, see, for example, the survey of Wormald~\cite{NW-survey}.

\subsection{Lower bound}

For completeness, let us briefly discuss the following known lower bound for the modularity of $\G_{n,d}$. It is known that a.a.s.\ for any $d \in \nat \setminus \{1,2\}$, $\G_{n,d}$ is Hamiltonian.  As pointed out in~\cite{Colin1}, one can use this fact to partition the graph such that it breaks the cycle into $\lceil \sqrt{n} \rceil$ paths of length at most $\lceil \sqrt{n} \rceil$. For this particular partition the edge contribution is $2/d - O(1/\sqrt{n})$ and the degree tax is $O(1/\sqrt{n})$. It follows then that a.a.s.
$$
q^*(\G_{n,d}) \ge \frac 2d - O(1/\sqrt{n}) = \frac {2+o(1)}{d}.
$$
(Our more general lower bound that holds for graphs with average degree $d$ implies the same---see Theorem~\ref{thm:avg_degree} for more.) Whereas this trivial lower bound could be sharp for $d=3$ it is definitely not the case for large $d$. As pointed out in~\cite{Colin2}, there exists a universal constant $c>0$ such that a.a.s.\ $q^*(\G_{n,d}) \ge c/\sqrt{d}$.

\subsection{Numerical upper bound}\label{sec:numerical}

The following straightforward lemma is useful for obtaining upper bounds for modularity of random $d$-regular graphs.

\begin{lem}\label{lem:upper_bound}
Consider any $d$-regular graph on $n$ vertices $G_{n,d}$. 
If no subset of $V(G_{n,d})$ of size $xn$ induces $y x n / 2$ edges with $y/d-x \ge U$, then 
$q^*(\G_{n,d}) < U$.
\end{lem}

\begin{proof}
For a given partition $\A = \{A_1, \ldots, A_k\}$ of the vertex set $V(G)$, let $x_i = |A_i|/n$ and $y_i = 2 |E(A_i)|/|A_i|$; that is, set $A_i$ has $x_in$ vertices and induces $y_i x_i n / 2$ edges. Then, taking into account the fact that for any $A\subseteq V(G)$ we have $\sum_{v\in A} \deg(v) = d|A|$, we can rewrite~(\ref{eq:q_A}) as
\begin{equation}\label{eq:q_A_sim}
q_{\A} = \sum_{i=1}^k x_i \left( \frac{y_i}{d} - x_i \right).
\end{equation}
As it is simply a weighted average, $q_{\A} \ge U$ would imply that there exists some set of size $xn$ that induces $y x n / 2$ edges, and $y/d-x \ge U$. 
So, the proof of the lemma is finished.
\end{proof}

To formulate the main theorem of this section, we need the following notation.
For a given $d \in \nat \setminus \{1,2\}$, let
\begin{align}
f(x,&y,d) := x(y/2-1) \log (x) + (1-x)(d-1)\log(1-x) + d \log (d) / 2  \label{eq:function_f} \\
& - xy \log (y) / 2 - x(d-y) \log(d-y) - (d-2xd+xy) \log (d-2xd+xy) /2. \nonumber
\end{align}
It will be clear once we establish the connection between function $f$ and random $d$-regular graphs, but it is straightforward to see that for any $x \in (0,1)$ we have $f(x,d,d) < 0$ (more precisely, its limit value) and $f(x,y,d) > 0$ for some $y \in (0,d)$. Indeed, for example note that $f(x,xd,d) = - x \log (x) + (x-1)\log(1-x) > 0$. Also, it is easy to see that $f(x,y,d)$ is continuous on $y \in (0,d)$.

Finally, let $\bar y = \bar y(x,d)$ be largest value of $y \in (0,d)$ such that $f(x,y,d)=0$; in particular, $f(x,y,d) < 0$ for any $y \in (\bar y, d)$. 

\begin{thm}\label{thm:U3}
Let $d \in \nat \setminus \{1, 2\}$ and $\eps >0$ be an arbitrarily small constant. Then a.a.s.
$$
q^*(\G_{n,d}) \le U_3 + \eps / d,
$$
where
$$
U_3 = U_3(d) := \sup_{x \in (0,1)} \left( \frac{\bar y(x,d)}{d} - x \right).
$$
\end{thm}

As usual, see Table~\ref{tab:upper_bound} for numerical values for small values of $d$. 


\begin{proof}
We prove below that the following property holds a.a.s.\ for $\G_{n,d}$. No set $A$ of size $xn$ (for any $x=x(n) \in (0,1)$) induces a graph with $yxn/2$ edges, where $\bar y(x,d) + \eps \le y \le d$ and $\bar y(x,d)$ is defined as above. Then Theorem~\ref{thm:U3} follows directly from Lemma~\ref{lem:upper_bound}.

Consider $\mathcal{G}_{n,d}$ for some $d \in \nat \setminus \{1,2\}$ and let $\eps >0$ be an arbitrarily small constant. Our goal is to show that the expected number of sets $S$ such that $|S|=xn$ and $e(S) = yxn/2$ with $y \ge \bar y(x,d) + \eps$ is $o(n^{-2})$. (For simplicity, we do not round numbers that are supposed to be integers either up or down; this is justified since these rounding errors are negligible to the asymptomatic calculations we will make.) This, together with the first moment principle, implies that a.a.s.\ no such set exists for any $x \in (0,1)$ and $y \in [\bar y(x,d) + \eps,d]$ (as there are $O(n)$ possible sizes of $S$ and $O(n)$ possible values of $e(S)$ that we need to consider). 

Let $x=x(n)$ and $y=y(n)$ be any functions of $n$ such that $0 < x < 1$ and $\bar y(x,d) + \eps < y < d$. Let $X(x,y)$ be the expected number of sets $S$ such that $|S|=xn$ and $e(S) = yxn/2$. Using the pairing model, it is clear that
\begin{eqnarray*}
X(x,y) &=& {n \choose xn} { dxn  \choose yxn } { d(1-x)n \choose (d-y)xn } ((d-y)xn)! \, M(yxn)  \\
&& \quad \cdot \ M( d(1-x)n - (d-y)xn) / M(dn),
\end{eqnarray*}
where $M(i)$ is the number of pairings of $i$ vertices, that is,
$$
M(i) = \frac {i!} {(i/2)! 2^{i/2}}.
$$
(Each time we deal with pairings, $i$ is assumed to be an even number.) After simplification we get
\begin{eqnarray*}
X(x,y) &=& n! \left( dxn \right)! \left( d(1-x)n \right)! \left( yxn \right)! (dn/2)! \ 2^{dn/2} \cdot \ \Bigg[ \left( xn \right)! \left( (1-x)n \right)! \left( yxn \right)! \\
&& \quad \quad \quad \left( (d-y)xn \right)! \left( \frac {yx}{2} n\right)! \ 2^{\frac{yx}{2} n} \left( \frac{d-2dx+yx}{2} n \right)! \ 2^{\frac{d-2dx+yx}{2} n} (dn)! \Bigg]^{-1}.
\end{eqnarray*}
Using Stirling's formula ($i! \sim \sqrt{2\pi i} (i/e)^i$) and focusing on the exponential part we obtain
$$
X(x,y) = \Theta( n^{-1} ) e^{f(x,y,d)n},
$$
where $f(x,y,d)$ is defined in~(\ref{eq:function_f}). It follows immediately from the definition of $\bar y(x,d)$ that $f(x,y,d) < 0$ is bounded away from zero for any pairs of integers $xn$ and $yxn/2$ under consideration, and so for any pair we get $X(x,y) = o(n^{-2})$ and the proof is finished.
\end{proof}

\subsection{Explicit but weaker upper bound}

Theorem~\ref{thm:U3} provides an upper bound that can be easily numerically computed for a given $d \in \nat \setminus \{1,2\}$. Next, we present a slightly weaker but an explicit bound that can be obtained using the expansion properties of random $d$-regular graphs that follow from their eigenvalues. In particular, it will imply that a.a.s.\ $q^*(\G_{n,d}) = O(1/\sqrt{d})$ and so $q^*(\G_{n,d}) \to 0$ as $d \to \infty$.

The adjacency matrix $A=A(G)$ of a given $d$-regular graph $G$ with $n$ vertices, is an $n \times n$ real and symmetric matrix. Thus, the matrix $A$ has $n$ real eigenvalues which we denote by $\lambda_1 \ge \lambda_2 \ge \cdots \ge \lambda_n$. It is known that certain properties of a $d$-regular graph are reflected in its spectrum but, since we focus on expansion properties, we are particularly interested in the following quantity: $\lambda = \lambda(G) = \max( |\lambda_2|, |\lambda_n|)$. In words, $\lambda$ is the largest absolute value of an eigenvalue other than $\lambda_1 = d$ (for more details, see the general survey~\cite{HLW} about expanders, or~\cite{AS}, Chapter 9).

The value of $\lambda$ for random $d$-regular graphs has been studied extensively. A major result due to Friedman~\cite{Fri} is the following:
\begin{lem}[\cite{Fri}]\label{lem:Fri}
For every fixed  $\varepsilon > 0$ and for $G\in \mathcal{G}_{n,d}$, a.a.s.\
$$
\lambda(G) \le 2 \sqrt{d-1}+ \eps.
$$
\end{lem}

We prove the following theorem.

\begin{thm}\label{thm:U3_explicit}
Let $d \in \nat \setminus \{1, 2\}$. Then, for any $d$-regular graph $G_{n,d}$ we have
$$
q^*(G_{n,d}) \le \frac{\lambda}{d}. 
$$
In particular, for random $d$-regular graphs a.a.s.
$$
q^*(\G_{n,d}) 
\le \frac {2}{\sqrt{d}}.
$$
\end{thm}

\begin{proof}
The second part of the theorem follows from Lemma~\ref{lem:Fri}, as for a random $d$-regular graphs a.a.s. $\frac{\lambda}{d} \le \frac { 2 \sqrt{d-1}+ \eps }{d} 
\le \frac {2}{\sqrt{d}}$ for sufficiently small $\eps >0$. Let us now show that $q^*(\G_{n,d}) \le \frac{\lambda}{d}$.

The number of edges $|E(S,T)|$ between sets $S$ and $T$ is expected to be close to the expected number of edges between $S$ and $T$ in a random graph of edge density $d/n$, namely $d|S||T|/n$. A small $\lambda$ (or large spectral gap) implies that this deviation is small. 
Namely, for our purpose here we will use the following lower estimate for $|E(S,V \setminus S)|$
\begin{equation} \label{eqn:bisection}
|E(S,V \setminus S)| \geq \frac{(d-\lambda)|S||V \setminus S|}{n}
\end{equation}
for all $S \subseteq V$. This is proved in~\cite{AM}, see also~\cite{AS}. Using this inequality we get immediately that for any $S$ of size $xn$ we have
\begin{equation} \label{eq:expansion}
e(S) = \frac {d|S| - |E(S,V \setminus S)|}{2} \le \frac {dxn - (d-\lambda) x(1-x)n}{2} = \frac {dx+\lambda(1-x)}{2} \cdot xn.
\end{equation}

So, a.a.s., in $\G_{n,d}$ no set $A$ of size $xn$ induces a graph with more than $yxn/2$ edges, where $y = dx + \lambda (1-x)$. Now the desired upper bound follows from Lemma~\ref{lem:upper_bound}.
\end{proof}

\bigskip

We have also tried several other ideas attempting to obtain a better upper bound. Unfortunately, they did not lead to improvements, therefore we place the discussion of these ideas to Appendix.

\section{Lower bounds in terms of average degree}\label{sec:lower}

In this section, we obtain some general lower bounds for modularity. In particular, the obtained bounds are useful for graphs with bounded average degree. In Section~\ref{sec:PA}, we apply these results to obtain a lower bound for the modularity of preferential attachment model (see Theorem~\ref{thm:PA_lower}).

Let us start with the analysis of trees. It was proven in~\cite{ModTrees} that trees with maximum degree $\Delta = o(\sqrt[5]{n})$ have asymptotic modularity 1. We generalize this result in two ways: first, we relax the condition on maximum degree; second, we allow our graphs to be disconnected, that is, we consider forests instead of trees. We prove the following theorem.

\begin{thm}\label{thm:forest}
Let $\{F_n \}$ be a sequence of forests, where $F_n$ has $n$ non-isolated vertices and the maximum degree $\Delta = \Delta(F_n)$. Then the following lower bound holds
$$
q^*(F_n) \ge 1 - 3\sqrt\frac{\Delta}{n}.
$$
\end{thm}

This theorem implies that if the maximum degree $\Delta(F_n) = o(n)$, then $q^*(F_n) = 1-o(1)$.
Note that it is also known that the asymptotic modularity of trees with maximum degree $\Delta = \Omega(n)$ is strictly less than 1~\cite{ModTrees}. Hence, the assumption $\Delta = o(n)$ cannot be eliminated. 

We further generalize the above theorem to all connected graphs and prove the following result.

\begin{thm}\label{thm:avg_degree}
Let $\{G_n\}$ be a sequence graphs, where $G_n$ is a connected graph on $n$ vertices with the maximum degree $\Delta = \Delta(G_n)$
and the average degree  $\bar d = \bar d (G_n)$.
Then  
$$ 
q^*(G_n) \ge \frac{2}{\bar d} - 3\sqrt\frac{\Delta}{n \bar d} - \frac{\Delta}{n \bar d}. 
$$ 
\end{thm}

The theorem implies that if $\bar d(G_n) \le D$ for some constant $D$ and $\Delta(G_n) = o(n)$, then $q^*(G_n) \ge \frac{2}{D} - o(1)$.
Note that for $\bar d = 2$ 
Theorem~\ref{thm:avg_degree} looks similar to Theorem~\ref{thm:forest}. However, there are two important differences: Theorem~\ref{thm:avg_degree} is not restricted to forests, but requires graphs to be connected.

Before we prove both theorems let us introduce some notation and the main lemma which we will use.

\begin{defn}
Let $G$ be a graph and let $A$ be any subset of its vertex set $V(G)$. We define $\vol_{G}(A) := \sum_{v \in A} \deg(v)$, where $\deg(v)$ is the degree of a vertex $v$ in $G$. We also use the notation $\vol_{G}(G'):=\vol_G(V(G'))$, where $G'$ is a subgraph of $G$.
\end{defn}

\begin{lem}\label{lem:forest_decompose}
For every connected graph $G$ with maximum degree $\Delta$ and every $h>0$ there exists a partition of the vertex set into connected parts $A_1, \ldots, A_k$ such that $\frac{h}{\Delta} -1 \le \vol_G(A_i) \le h$. 
for all $1 \le i \le k$.
\end{lem}

\begin{proof}
For a graph $G$ let us consider its spanning tree $T$ and decompose it, by removing some edges, into subtrees $T_1, \ldots, T_k$ such that $\frac{h}{\Delta} -1 \le \vol_G(T_i) \le h$ for each $1\le i \le k$.
The way we do this decomposition is in a sense similar to the algorithm \textit{greedy-decompose}$_{\le h}$ from~\cite{ModTrees}. Namely, we first redefine a notion of a centroid edge of a subtree $T'$ of the initial tree $T$. 

\begin{defn}
The removal of any edge from a tree $T'$ splits $T'$ into two parts $T^1$ and $T^2$. A centroid edge of $T$ is an edge chosen to maximize $\min\{\vol_{G}(T^1), \vol_{G}(T^2)\}$. 
\end{defn}

Our algorithm is the following: as long as our forest contains a tree $T'$ with $\vol_{G}(T')  > h$, it finds a centroid edge $e$ of $T'$ and removes it. After this decomposition, we obtain trees $T_{1}, \ldots, T_{k}$ and we set $A_i = V(T_i)$ for $1 \le i \le k$. 

Obviously, for each $i$ we have $\vol_{G}(A_i) \le h$. Let us show that we also have $\vol_G(A_i) \ge \frac{h}{\Delta} -1$. 
Consider any step of our decomposition procedure. 
We take a tree $T'$ with $\vol_{G}(T') = h'>h$, remove its centroid edge $e$, and obtain two trees $T^1$ and $T^2$. Without loss of generality we may assume that $\vol_{G}(T^1)\le \vol_{G}(T^2)$. Let $s = \vol_{G}(T^1)$, $s \le h'/2$.
Let $x$ be the vertex incident with $e$ and belonging to $T^2$. For every edge $e'$ incident with $x$, for the part $T''$ of $T'-e'$ not containing $x$ we have $\vol_{G}(T'') \le s$ (otherwise $e$ is not a centroid edge). 
As $x$ has degree at most $\Delta$, we have $h' \le 
\Delta s + \Delta$  (at most $s$ for each of the $ \le \Delta$ parts plus the degree of $x$ itself). 
So, $s \ge \frac{h'-\Delta}{\Delta} > \frac{h}{\Delta}-1$. This proves that $\vol_G(A_i) \ge \frac{h}{\Delta} -1$ and completes the proof of the lemma.
\end{proof}

\-

Now, we are ready to prove Theorem~\ref{thm:avg_degree} and Theorem~\ref{thm:forest}.

\begin{proof} \emph{(Proof of Theorem~\ref{thm:avg_degree}.)}
Let us take $h = \sqrt{n\Delta\bar d}+\Delta$ and partition $V(G_n)$ into $A_1, \ldots, A_k$ according to Lemma~\ref{lem:forest_decompose}. To obtain the desired lower bound, we estimate $q_{\A}$ for $\A = \{A_1, \ldots, A_k\}$.
We first deal with the edge contribution.
As stated in Lemma~\ref{lem:forest_decompose}, we have $\vol_{G_n}(A_i) > \frac{h}{\Delta} -1$ for all $i$. 
Also, $\sum_{i} \vol_{G_n}(A_i) =\vol_{G_n}(G_n) = n\bar d$.  Therefore, $k \le n\bar d/(\frac{h}{\Delta} -1)$.
The number of intracluster edges in the spanning tree is $n-k$, and clearly this is the lower bound for $\sum_{A \in \A} e(A)$. 
Finally, 
$$
\sum_{A \in \A} \frac {e(A)}{|E(G_n)|} 
\ge \frac{n-k}{n\bar d/2} 
\ge \frac {2}{\bar d}-\frac{2}{\frac{h}{\Delta} -1}
= \frac 2 {\bar d} - 2\sqrt\frac{\Delta}{n\bar d}.
$$

It remains to estimate the degree tax.
Recall that a $\vol_{G_n}(A_i) \le h$ for all $i$ and $\sum_i \vol_{G_n}(A_i) = n\bar d$. Therefore,
$$
\sum_{A \in \A} \frac { \vol_{G_n}^2(A) }{4|E(G_n)|^2}
= \frac{h n \bar d }{n^2 \bar d^2}
= \frac{h }{n \bar d}
= \sqrt\frac{\Delta}{n \bar d} + \frac{\Delta}{n \bar d}\,.
$$
and so the proof is finished.
\end{proof}

\bigskip

\begin{proof} \emph{(Proof of Theorem~\ref{thm:forest}.)}
This proof is similar to the previous one. 
Let us fix $h = \sqrt{\Delta n}$.
The idea is to partition $V(F_n)$ into $A_1, \ldots, A_k$ such that for each $i$: $\vol_{F_n}(A_i) \le h$ and a subgraph induced by $A_i$ is a tree. 
Our forest $F_n$ may already contain trees $T_1, \ldots, T_\ell$ with $\vol_{F_n}(T_i) \le h$.
Let us denote the corresponding vertex sets by $A_1, \ldots, A_l$.
We decompose the remaining trees according to 
Lemma~\ref{lem:forest_decompose} (applied to each tree separately) into $A_{l+1}, \ldots, A_k$.






Now we have the partition $\A = \{A_1, \ldots, A_k\}$ of the vertex set $V(F_n)$. In order to estimate $q_{\A}$ we first consider the edge contribution.
According to Lemma~\ref{lem:forest_decompose}, $\vol_{F_n}(A_i) \ge \frac{h}{\Delta} -1$ for $\ell+1 \le i \le k$. 
Therefore, it is easy to show that for each intercluster edge we can find at least $\frac{h}{2\Delta} - 1$ inracluster edges. Hence, 
$$
\sum_{A \in \A} \frac {e(A)}{|E(F_n)|} 
\ge 1 - \frac{1}{\frac{h}{2\Delta}}
= 1 - 2\sqrt\frac{\Delta}{n}\,.
$$
It remains to estimate the degree tax.
Recall that $\vol_{F_n}(A_i) \le h$ for $1 \le i \le k$ and $\sum_i \vol_{F_n}(A_i) \ge n$. Therefore,
$$
\sum_{A \in \A} \frac { \vol_{F_n}^2(A) }{4|E(F_n)|^2} 
\le \frac{h \vol_{F_n}(F_n)}{ \vol_{F_n}(F_n)^2}
\le \frac{h}{n} = \sqrt\frac{\Delta}{n}\,.
$$
and so the proof is finished.
\end{proof}


\section{The Preferential Attachment model}\label{sec:PA}

\subsection{Lower bound}

The following theorem easily follows from the results of the previous section.

\begin{thm}\label{thm:PA_lower}
For any $\varepsilon>0$ a.a.s.
$$
q^*(G_m^n) \ge \frac{1}{m} - O\left(n^{-1/4+\varepsilon}\right) = \frac {1}{m} - o(1).
$$
\end{thm}

\begin{proof}
Let $\varepsilon>0$.
It is well-known that a.a.s.\ $\Delta(G_m^n) =  O\left(n^{\frac{1}{2}+ 2 \varepsilon}\right)$
(see, e.g., \cite{HighDegree} and Theorem~17 in \cite{Math_Results}). Also, clearly the average degree of $G_m^n$ is at most $2m$ (it can be less due to the removal of loops and multiple edges). In addition, for $m\ge 2$ a.a.s.\ $G_m^n$ is connected~\cite{BR}. So, the statement of Theorem~\ref{thm:PA_lower} follows directly from Theorems~\ref{thm:forest} and~\ref{thm:avg_degree}.
\end{proof}

\bigskip

We would like to remark that the obtained lower bound holds for many other models of complex networks. For example, it holds for the Random Apollonion Network~\cite{RAN} (in this case $m=3$) or for the Buckley-Osthus model~\cite{Buckley_Osthus} (with slightly corrected error term).

As in the case of random $d$-regular graphs, it is natural to conjecture that the above lower bound is not sharp. Let $c \in (0,1)$ and consider the following partition: $A_1 = \{v_1, \ldots, v_{cn} \}$, $A_2 = V(G_m^n) \setminus A_1 = \{ v_{cn+1}, \ldots, v_n\}$. Using martingales, it is possible to show that a.a.s.\ $\sum_{v \in A_1} \deg(v, n) \sim 2mn\sqrt{c}$ (and so $\sum_{v \in A_2} \deg(v, n) \sim 2mn(1-\sqrt{c})$); see Lemma~\ref{lem:total_weight_specific} below. Clearly, $e(A_1) = mcn$ and so a.a.s.\ $e(A_1,A_2) \sim 2mn(\sqrt{c}-c)$ and $e(A_2) \sim mn(1+c-2\sqrt{c})$. The edge contribution and the degree tax are then both asymptotic to $1+2c-2\sqrt{c}$. Not surprisingly, such partition cannot be used to get a non-trivial lower bound for the modularity but, similarly to the situation for random $d$-regular graphs, we may try to use it as a starting point to get slightly better partition. The basic idea is very simple: one can start with a given partition (or partition the vertices randomly into two classes), and if a vertex has more neighbours in the other class than in its own, then we randomly decide whether to shift it to the other class or leave it where it is. This approach proved to be useful to get a bound for the bisection width in random $d$-regular graphs~\cite{Alon_expansion} which, in turn, yields a lower bound for the modularity~\cite{Colin2}. In the proceeding version of this paper~\cite{proceedings} we promised to investigate this approach. However, the following turns out to be slightly easier to do. 

\medskip

We will use the following standard martingale tool: the \emph{Hoeffding-Azuma inequality}; for more details, see, for example,~\cite{JLR}. Let $X_0, X_1, \ldots$ be a martingale. Suppose that there exist $c_1, c_2, \ldots,c_n >0$ such that $|X_k - X_{k-1}| \le c_k$ for each $1\le k \le n$. Then, for every $x >0$,
\begin{equation}\label{eq:HA-inequality1}
\pr [ X_n > \ex X_n + x ] \le \exp \left(-\frac{x^2}{2\sum_{k=1}^n {c_k}^2}\right).
\end{equation}

The Hoeffding-Azuma inequality can be generalized to include random variables close to martingales. One of our proofs, proof of Lemma~\ref{lem:total_weight_specific}, will use the supermartingale method of Pittel et al.~\cite{Pittel99}, as described in~\cite[Corollary~4.1]{Wormald-DE}. Let $X_0, X_2, \ldots, X_n$ be a sequence of random variables. Suppose that there exist $c_1, c_2, \ldots,c_n >0$ and $b_1, b_2, \ldots,b_n >0$ such that 
$$
|X_k - X_{k-1}| \le c_k \quad \text{ and } \quad \ex [X_k - X_{k-1} | X_{k-1} ] \le b_k
$$ 
for each $1\le k \le n$. Then, for every $x >0$,
\begin{equation}\label{eq:HA-inequality2}
\pr \left[ \text{For some $t$ with $0 \le t \le n$: } X_t - X_0 > \sum_{k = 1}^{t} b_k + x \right] \le \exp \left(-\frac{x^2}{2\sum_{k=1}^n {c_k}^2}\right).
\end{equation}

Let us now prove the following lemma.

\begin{lem}\label{lem:total_weight_specific} 
Fix any constant $c\in (0,1)$ and $m\in\ent_{\ge0}$. The following property holds a.a.s.\ for $G^n_m$. For any $s$, $cn \le s \le n$, 
$$
\Big| Y_s - 2mn \sqrt{cs/n} \Big| \le (mn)^{2/3},
$$
where $Y_s := \sum_{w \in [cn]} \deg(w,s) $.
\end{lem}

\begin{proof}
In view of the identification between the models $G_m^n$ (on the vertex set $1, 2, \ldots, n$) and $G_1^{mn}$ (on the vertex set $1', 2', \ldots, mn'$), it will be useful to investigate the following random variable instead of $Y_s$: for $m\lfloor cn\rfloor \le t \le mn$, let 
$$
X_t = \sum_{j' \in [cmn]} \deg(j', t).
$$
Clearly, $Y_s = X_{sm}$. It follows that $X_{m\lfloor cn\rfloor} = Y_{\lfloor cn\rfloor} = 2m\lfloor cn\rfloor$. Moreover, for $m\lfloor cn\rfloor < t \le mn$,
$$
X_t =
\begin{cases}
X_{t-1}+1 & \text{with probability } \frac {X_{t-1}}{2t-1}, \\
X_{t-1} & \text{otherwise}.
\end{cases}
$$
The conditional expectation is given by
$$
\ex \left( X_t | X_{t-1} \right) = (X_{t-1}+1) \cdot \frac {X_{t-1}}{2t-1} + X_{t-1} \left(1- \frac {X_{t-1}}{2t-1} \right) = X_{t-1} \left(1+ \frac {1}{2t-1} \right).
$$
Taking expectation again, we derive that
$$
\ex {X_t} = \ex { X_{t-1} } \left(1+ \frac {1}{2t-1} \right).
$$
Hence, it follows that 
\begin{multline*}
\ex (Y_s) = \ex ( X_{sm})  = 2m\lfloor cn\rfloor \prod_{i=m\lfloor cn\rfloor+1}^{sm}  \left(1+ \frac {1}{2i-1} \right) 
= 2m\lfloor cn\rfloor \frac{\mathrm{\Gamma}(sm+1)\mathrm{\Gamma}(m\lfloor cn\rfloor+1/2)}{\mathrm{\Gamma}(sm+1/2)\mathrm{\Gamma}(m\lfloor cn\rfloor+1)}
\\\sim 2cmn \left( \frac {sm}{cmn} \right)^{1/2} = 2mn \sqrt{cs/n}.
\end{multline*}

In order to transform $X_t$ into something close to a martingale (to be able to apply the generalized Azuma-Hoeffding inequality~(\ref{eq:HA-inequality2})), we set for $m \lfloor cn \rfloor \le t \le mn$
$$
Z_t = X_t - 2m \lfloor cn \rfloor - \sum_{k=m \lfloor cn \rfloor + 1}^{t} \sqrt{cmn/k}
$$ 
(note that $Z_{m \lfloor cn \rfloor} = 0$) and use the following stopping time
$$
T = \min \left\{ t > m \lfloor cn \rfloor : X_t > 2 \sqrt{t cmn} + t^{2/3} \text{ or } t = mn \right\}.
$$
Indeed, we have for $m \lfloor cn \rfloor < t \le mn$
$$
\ex \left( Z_t - Z_{t-1} ~~|~~ Z_{t-1} \right) = \frac {X_{t-1}}{2t-1} - \sqrt{cmn/t} \le (1/2+o(1)) t^{-1/3} < 0.51 t^{-1/3},
$$
provided $t \le T$, and $|Z_t - Z_{t-1}| \le 1$ as $t > cmn$.
Let $t \wedge T$ denote $\min\{t,T\}$. We apply the generalized Azuma-Hoeffding inequality~(\ref{eq:HA-inequality2}) to the sequence $(Z_{t \wedge T} : m \lfloor cn \rfloor \le t \le mn)$, with $c_t = 1$, $b_t = 0.51 t^{-1/3}$ and $x = 0.1 t^{2/3}$, to conclude that a.a.s.\ for all $t$ such that $m \lfloor cn \rfloor \le t \le mn$
$$
Z_{t \wedge T} - Z_{m \lfloor cm \rfloor} = Z_{t \wedge T} \le  \sum_{k\le t} b_k + x \le 0.77 t^{2/3} + 0.1 t^{2/3} \le 0.9 t^{2/3}.
$$

To complete the proof we need to show that a.a.s. $T = mn$. The events asserted by the equation hold a.a.s.\ up until time $T$, as shown above. Thus, in particular, a.a.s.
\begin{eqnarray*}
X_T &=& Z_T + 2m\lfloor cn \rfloor + \sum_{k=m \lfloor cn \rfloor + 1}^{T} \sqrt{cmn/k} \\
&\le& 0.9 T^{2/3} + 2mcn + \sqrt{cmn} \int_{mcn}^{T} 1/\sqrt{k} \; dk + O(1) \\
&<& 2\sqrt{Tcmn} + T^{2/3},
\end{eqnarray*}
which implies that $T = mn$ a.a.s. In particular, it follows that a.a.s., for any $cn \le s \le n$, $Y_{s} = X_{ms} < 2 mn \sqrt{cs/n} + (mn)^{2/3} = 2 mn \sqrt{cs/n} + o(n)$. The lower bound can be obtained by applying the same argument symmetrically to $(-Z_{t \wedge T} : m \lfloor cn \rfloor \le t \le mn)$, and so the proof is finished.
\end{proof}

\bigskip

Now, we are ready to prove the following, stronger, lower bound.

\begin{thm}\label{thm:PA_lower-stronger}
A.a.s.:
$$
q^*(G_m^n) \ge \ex \left( \Big| \Bin (m,1/2) - m/2 \Big| \right) / m + o(1).
$$
That is, a.a.s.
$$
q^*(G_m^n) \ge 
\begin{cases}
(2^{1-m} / m)  \sum_{i=1}^{m/2} i \binom{m}{m/2+i} & \text{ if $m$ is even}, \\
(2^{1-m} / m)  \sum_{i=1}^{(m+1)/2} (i-1/2) \binom{m}{(m-1)/2+i} & \text{ if $m$ is odd}, \\
\end{cases}
$$
In particular, a.a.s.\ $q^*(G_m^n) = \Omega \left( 1/\sqrt{m} \right)$.
\end{thm}

Before we prove the theorem, let us present numerical values for a few values of $m$: $L_1=L_1(m)=1/m$ is the lower bound following from Theorem~\ref{thm:PA_lower} and $L_2=L_2(m)$ is the lower bound from Theorem~\ref{thm:PA_lower-stronger}; see Table~\ref{tab:lower_bound_PA}. Large degree tax hidden in $L_2$ makes this bound weaker for small values of $m \le 6$; for larger values $L_2$ is better than $L_1$. 

\begin{table}[h]
\begin{center}
\caption{Lower bounds for $q^*(G_m^n)$}
  \begin{tabular}{ | c || c | c | c | c | c | c | }
    \hline
    $m$ & $7$ & $8$ & $9$ & $10$ & $100$ & $1000$ \\ \hline \hline
    $L_1$ & 0.142 & 0.125 & 0.111 & 0.100 & 0.0100 & 0.0010 \\
    $L_2$ & 0.156 & 0.136 & 0.136 & 0.123 & 0.0397 & 0.0126 \\
    \hline
  \end{tabular}
\label{tab:lower_bound_PA} 
\end{center}
\end{table}

\begin{proof}
Let $\eps > 0$ be any constant. Let us start with generating $G_m^{\eps n}$; vertices from $[\eps n / 4]$ are coloured red and vertices from $[\eps n] \setminus [\eps n / 4]$ are coloured blue. We will continue generating $G_m^n$, colouring vertices red or blue (one by one, as they are introduced in the process), depending on how many of their neighbours are of each colour. We want to control the sum of degrees of vertices in each colour; that is, the following random variable 
$$
Y_{t} := \sum_{w \in [t], w \text{ is red }} \deg(w, t).
$$
The colouring process depends on the parity of $m$. If $m$ is even, we colour vertex $t \in [n] \setminus [\eps n]$ red if more than $m/2$ neighbours (in $G_m^t$) are red. If the number of red neighbours is precisely $m/2$, we colour it red with probability $1/2 + p_t$, where $p_t = p_t(Y_{t-1}) = o(1)$ will be determined soon. Otherwise, $t$ is coloured blue. If $m$ is odd, the process is slightly different. If the number of red neighbours is more than $(m+1)/2$, we colour it red. If it is $(m+1)/2$ or $(m-1)/2$, we colour it red with probability $1 - q_t$ and, respectively $r_t$, where $q_t + r_t = q_t(Y_{t-1}) + r_t(Y_{t-1}) = o(1)$. Otherwise, $t$ is coloured blue. The arguments for both cases are almost identical so we assume now that $m$ is even; it will be clear what needs to be adjusted for odd value of $m$. In both situations, our hope is that the two graphs, induced by red and blue vertices, will be dense.

It follows from Lemma~\ref{lem:total_weight_specific} that a.a.s.\ $|Y_{\eps n} - m\eps n| \le (m\eps n)^{2/3}$, so we may assume that this inequality holds. This time we use the following stopping time
$$
T = \min \left\{ t > \eps n : | Y_t - mt | > 2(mt)^{2/3} \text{ or } t = n \right\}.
$$
Arguing as in the previous lemma, we get that
\begin{eqnarray*}
\ex (Y_t - Y_{t-1} ~~|~~ Y_{t-1}) &=& Y_{t-1} \left( \frac {m}{2mt} + O \left( t^{-2} \right) \right) + m \Bigg( \pr \left( \Bin \left( m, \frac {Y_{t-1} + O(1)}{2mt} \right) > m/2 \right) \\
&& \quad \quad \quad \quad + \left( \frac 12 + p_t \right) \pr \left( \Bin \left( m, \frac {Y_{t-1} + O(1)}{2mt} \right) = m/2 \right) \Bigg) \\ 
&=& \frac {m}{2} + O \left( t^{-1/3} \right) + m \Bigg( \pr \left( \Bin \left( m, \frac {1}{2} + O(t^{-1/3}) \right) > m/2 \right) \\
&& \quad \quad \quad \quad + \left( \frac 12 + p_t \right) \pr \left( \Bin \left( m, \frac {1}{2} + O(t^{-1/3}) \right) = m/2 \right) \Bigg),
\end{eqnarray*}
provided that $t<T$. Since
$$
\pr \left( \Bin \left( m, \frac {1}{2} + O(t^{-1/3}) \right) = i \right) = \pr \left( \Bin \left( m, \frac {1}{2} \right) = i \right) \left( 1 + O(t^{-1/3}) \right)
$$
and
$$
\pr \left( \Bin \left( m, \frac {1}{2} \right) > m/2 \right) + \frac 12 \ \pr \left( \Bin \left( m, \frac {1}{2} \right) = m/2 \right) = \frac 12,
$$
we get that
\begin{eqnarray*}
\ex (Y_t - Y_{t-1} ~~|~~ Y_{t-1}) &=& m + p_t \ \pr \left( \Bin \left( m, \frac {1}{2} \right) = m/2 \right) + O(t^{-1/3}).
\end{eqnarray*}
Since $\pr \left( \Bin \left( m, 1/2 \right) = m/2 \right) = \binom{m}{m/2} 2^{-m} = \Theta(1)$, we can adjust $p_t = p_t(Y_{t-1}) = O(t^{-1/3})$ so that $\ex (Y_t - Y_{t-1} ~~|~~ Y_{t-1}) = m$; that is, the sequence of random variables $Z_t = (Y_t - Y_{\eps n}) - m(t-\eps n)$ is a martingale. It follows from the classic Hoeffding-Azuma inequality~(\ref{eq:HA-inequality1}), applied to $Z_t$ with $c_t = m$ and $x = (mn)^{2/3}$, that a.a.s., for each $\eps n \le t \le n$,
$$
|Y_t - mt| \le |Z_t| + |Y_{\eps n} - m \eps n| \le x + (m\eps n)^{2/3} \le 2 (mn)^{2/3} = o(t).
$$

The rest of the proof is straightforward. We partition the vertex set of $G_n^m$ into red and blue vertices. The degree tax is a.a.s.\ 
$$
\frac {Y_n^2 + (2mn - Y_n)^2}{4 (mn)^2 } = \frac 12 + o(1).
$$
It remains to estimate the edge contribution. Clearly, the process guarantees that at least half of the edges are within the two clusters. However, we will do slightly better than that. For any $i \in [m/2]$, with probability asymptotic to $2 \pr (\Bin(m,1/2)=m/2+i)$, at any point of the process we add $m/2+i$ edges to some cluster; $m/2$ edges are added with probability asymptotic to $\pr (\Bin(m,1/2)=m/2)$. Hence, the expected number of edges added to some cluster is asymptotic to 
$$
m/2 + \ex \left( \Big| \Bin (m,1/2) - m/2 \Big| \right) = m/2 + 2 \sum_{i=1}^{m/2} i \binom{m}{m/2+i} 2^{-m}.
$$
The expected edge contribution is then asymptotic to 
$$
1/2 + \ex \left( \Big| \Bin (m,1/2) - m/2 \Big| \right) (n-\eps n) / (mn).
$$
Finally, one can bound the edge contribution (independently, from above and from below) by the sum of independent random variables, and use Chernoff bound to get a concentration. It follows that a.a.s.\ 
$$
q^*(G_m^n) \ge \ex \left( \Big| \Bin (m,1/2) - m/2 \Big| \right) (1-\eps) / m + o(1),
$$
and the result holds after taking $\eps \to 0$ sufficiently slowly.

Finally, some elementary calculations show that for any $t \in [0,m/8]$, we have
$$
\pr \left( \Bin (m,1/2) \ge m/2  + t \right) \ge \frac {e^{-16t^2/m}}{15};
$$
see, for example,~\cite{book}. (More general and precise bounds can be found in~\cite{Feller}.) It follows that a.a.s.\ $q^*(G_m^n) \ge 2 (\sqrt{m}/8) e^{-1/4} / (15 m) + o(1) = \Omega(1/\sqrt{m})$, and the proof is finished.
\end{proof}

\subsection{Upper bound}

Recall that the \emph{edge expansion} $\rho=\rho(G)$ of a graph $G$ is defined as follows:
$$
\rho = \min_{S \subset V(G), |S| \le |V|/2} \frac {e(S,V \setminus S)}{|S|}.
$$
In~\cite{Mihail} it was shown that a.a.s.\ $\rho(G^n_m) \ge \alpha$, provided that $2(m-1)-4\alpha-1>0$. In other words, for any $\eps>0$ we have that a.a.s.
$$
\rho(G^n_m) \ge \frac {m}{2} - \frac{3+\eps}{4}.
$$
Using this observation one can easily obtain a non-trivial upper bound for $q^*(G_m^n)$.

Let $\eps > 0$ be an arbitrary small constant. Consider any partition $\A = \{A_1, \ldots, A_k\}$ of the vertex set $V(G_m^n)$. If $|A_i| > n/2$ for some $i$, then the degree tax is at least
$$
\frac { (\sum_{v \in A_i} \deg(v, n))^2}{4 |E(G_m^n)|} \ge \frac {(m|A_i|)^2}{4(mn)^2} = \frac {1}{16}. 
$$
On the other hand, if $|A_i| \le n/2$ for all $i$, then a.a.s.\ the number of edges between parts is equal to 
$$
\frac 12 \sum_{i=1}^k e(A_i, V \setminus A_i) \ge \frac 12 \sum_{i=1}^k \rho |A_i| = \frac {\rho n}{2} \ge \left( \frac {m}{4} - \frac{3+\eps}{8} \right) n,
$$
and so the edge contribution is a.a.s.\ at most 
$$
1 - \left( \frac {1}{4} - \frac{3+\eps}{8m} \right) = \frac {3}{4} + \frac{3+\eps}{8m} \le \frac {15+\eps}{16},
$$
for any $m \ge 2$. Therefore, the following result holds. 

\begin{thm}\label{thm:PA_upper}
For any $\varepsilon>0$ a.a.s.
$$
q^*(G_2^n) \le \frac {15+\eps}{16}.
$$
Moreover, for any $m \ge 3$ a.a.s.
$$
q^*(G_m^n) \le \frac {15}{16}.
$$
\end{thm}

Some stronger expansion properties were recently obtained in~\cite{Frieze}. However, whereas they presumably could be used to obtain some small improvements for an upper bound of $q^*(G_m^n)$ (for specific values of $m$), we do not know how to show that $q^*(G_m^n) \to 0$ as $m \to \infty$. Perhaps $q^*(G_m^n) = \Theta(1/\sqrt{m})$ as in the case of random $(2m)$-regular graphs?


\section{The Spatial Preferential Attachment model}\label{sec:SPA}

Consider $G_{n}=(V_{n},E_{n})$, a graph generated by the SPA model. As the modularity is defined for undirected graphs, we consider $\hat{G}_n$ that is a graph obtained from $G_n$ by replacing each directed edge $(u,v)$ by undirected edge $uv$. (As edges in $G_n$ are always from `younger' to `older' vertices, there is no problem with generating multigraph; $\hat{G}_n$ is a simple graph.) Let us recall that $V_n \subseteq S$ where $S$ is the unit hypercube $[0,1]^m$. We will use the geometry of the model to obtain a suitable partition that yields high modularity of $G_n$.  The following properties (proved many times; see, for example,~\cite{spa1,spa2}) are the only properties of the model that will be used in the proof: a.a.s.\ for every pair $i,t$ such that $1 \le i \le t \le n$ we have that
\begin{eqnarray}
\deg^-(v_i, t) &=& O \Big( (t/i)^{pA_1} \log^2 n \Big), \label{eq:degree} \\
\deg^+(v_i, t) &=& O \Big( \log^2 n \Big), \label{eq:outdegree}
\end{eqnarray}
and $|E(G_n)| = \Theta(n)$. Since we aim for a result that holds a.a.s., we may assume in the proof below that these properties hold deterministically. Now, we are ready to state our result for the SPA model.

\begin{thm}\label{thm:spa_mod}
Let $p\in(0,1]$, $A_1, A_2 >0$, and suppose that $pA_1 < 1$. Then, the following holds a.a.s.:
$$
q^*(\hat{G}_n) = 1 - O \left( n^{\max \{-1/m,-1+pA_1\} / 2} \log^{9/2} n \right) = 1 - o(1).
$$
\end{thm}

\begin{proof}
Let $\omega = \left[n^{\min \{1/m,1-pA_1\} / 2} \log n^{-1/2}\right]$. Note that $\omega \ge n^{\eps}$ for some $\eps>0$ that depends on the parameters of the model. Let us partition the space $S$ into $\omega$ parts as follows: for each integer $1 \le r \le \omega$,
$$
S_r = \left\{ s = (s_1, \ldots, s_m) \in S : \frac {r-1}{\omega} \le s_1 < \frac {r}{\omega} \right\}.
$$
This partition of $S$ naturally gives us a partition $\A$ of the vertex set: for each $1 \le r \le \omega$, $A_r = V_n \cap S_r$. We will show that a.a.s.\ 
$$
q_{\A}(\hat{G}_n) = 1 - O \left( n^{\max \{-1/m,-1+pA_1\} / 2}  \log^{9/2} n \right),
$$
which will finish the proof as $q^*(\hat{G}_n) \ge q_{\A}(\hat{G}_n)$ and always $q^*(\hat{G}_n) \le 1$.

First, let us start with estimating the edge contribution. In order to do that, we need to estimate the number of edges between different parts. So, let us focus on any part $A_r$. We will investigate how many \emph{bad} edges in $G_n$ connect vertices outside of $A_r$ with vertices inside $A_r$ by counting (independently) bad edges directed to vertices of similar age. (Note that for convenience we consider here  directed graph $G_n$ instead of $\hat{G}_n$.)  For a given integer $k$ such that $0 \le k \le \lfloor \log n \rfloor$, let
\begin{eqnarray*}
V^{(k)} &=& \{ v_i \in V_n : e^{k} \le i < \min \{e^{k+1}, n+1\} \},\\
E^{(k)} &=& \{ (v_j, v_i) \in E_n : v_i \in V^{(k)}, v_j \in V_n, \text{ and } i < j \le n \} \\
C^{(k)} &=& \{ (v_j, v_i) \in E_n : v_i \in V^{(k)}\cap A_r, v_j \in V_n \setminus A_r, \text{ and } i < j \le n \} \subseteq E^{(k)}.
\end{eqnarray*}
It is clear that $\{ V^{(k)} : 0 \le k \le \lfloor \log n \rfloor \}$ and $\{ E^{(k)} : 0 \le k \le \lfloor \log n \rfloor \}$ are partitions of the vertex set and the edge set (both in $\hat{G_n}$ and $G_n$), respectively, and so $\{ C^{(k)} : 0 \le k \le \lfloor \log n \rfloor \}$ is a partition of the bad edges we want to count. It remains to estimate the size of $C^{(k)}$ for a given value of $k$.

Fix $0 \le k \le \lfloor \log n \rfloor$, and let us concentrate on any $v_i \in V^{(k)}$. It follows from~(\ref{eq:degree}) that the maximum volume of a sphere of influence of $v_i$ is $O(i^{-1} \log^2 n) = O(e^{-k} \log^2 n)$ (during the whole process) and so the maximum radius of influence of $v_i$ is $O((e^{-k} \log^2 n)^{1/m})$. Therefore, if there is an edge in the cut directed to $v_i=(s_1,\ldots,s_m)$, then $v_i$ must fall not only into $A_r$ but also into a strip within distance $O((e^{-k} \log^2 n)^{1/m})$ from one of the two cutting hyperplanes separating $A_r$ from the neighbouring parts; that is, $|s_1- \frac {r-1}{\omega}|=O((e^{-k}\log^2 n)^{1/m})$ or $|s_1- \frac {r}{\omega}|=O((e^{-k}\log^2 n)^{1/m})$. Since $|V^{(k)}|=O(e^k)$, we get that
$$
O((e^{-k} \log^2 n)^{1/m}) \cdot |V^{(k)}| = O(e^{k(1-1/m)} (\log n)^{2/m})
$$
vertices of $V^{(k)}$ are expected to appear in these two strips during the whole process. Hence, it follows from Chernoff bound that with probability at least $1-\exp(-\Theta(\log^2 n))$ there are $O(e^{k(1-1/m)} \log^2 n)$ vertices in these strips at the end of the process. Note that the exponent of $\log n$ has changed from $2/m$ to $2$ in order to guarantee the claimed upper bound is at least $\log^2 n$ which is required for a bound to hold with the desired probability. Using~(\ref{eq:degree}) one more time, we get that all vertices introduced in this time period have (final) in-degree $O((n/e^k)^{pA_1} \log^2 n)$. Hence, there are
$$
|C^{(k)}| = O\left( (e^{k(1-1/m)} \log^2 n) \cdot (n/e^k)^{pA_1} \log^2 n \right) = O \left(n^{pA_1} e^{k(1-1/m - pA_1)} \log^4 n \right)
$$
edges in the cut with probability at least $1-\exp(-\Theta(\log^2 n))$ and so this property holds a.a.s.\ for all parts $A_r$ and all values of $k$. It follows that a.a.s.\ the number of bad edges involving $A_r$ is at most
\begin{align*}
\xi_r &= \sum_{k=0}^{\lfloor \log n \rfloor} |C^{(k)}| =
\sum_{k=0}^{\lfloor \log n \rfloor} O(n^{pA_1} e^{k(1-1/m - pA_1)} \log^4 n) \\
& \le \left\{
    \begin{array}{ll}
      \log n \cdot O \left(n^{pA_1} n^{1-1/m - pA_1} \log^4 n \right), & \hbox{if $pA_1 < 1-1/m$;} \\
      \log n \cdot O(n^{pA_1} \log^4 n), & \hbox{otherwise,}
    \end{array}  \right. \\
&=O(n^{\max\{1-1/m,pA_1\}} \log^5 n).
\end{align*}
Finally, we get an estimate for the edge contribution: a.a.s.
\begin{align*}
\sum_{r = 1}^{\omega} \frac{e(A_r)}{|E(G_n)|} &= 1 - \sum_{r = 1}^{\omega} \frac{\xi_r}{|E(G_n)|} = 1 - \frac {\omega \cdot O(n^{\max\{1-1/m,pA_1\}} \log^5 n)}{\Theta(n)} \\
&= 1 - O \left( n^{\max \{-1/m,-1+pA_1\} / 2} \log^{9/2} n \right).
\end{align*}


It remains to estimate the degree tax. In order to do that we need to, for a given $r$ under consideration, estimate $\sum_{v \in A_r} \deg(v)$ in $\hat{G}_n$; that is, $\sum_{v \in A_r} (\deg^-(v)+\deg^+(v))$ in $G_n$. As before, we partition the vertices of $A_r$ into sets containing vertices of similar age. Let $k_0$ be the largest integer $k$ such that $(k-1)\omega \log^2 n < n$. Clearly, $k_0 = O(n / (\omega \log^2 n)$. This time, for a given integer $k$ such that $1 \le k \le k_0$, let
$$
V^{(k)} = \{ v_i \in V_n : (k-1) \omega \log^2 n  < i \le \min \{ k \omega \log^2 n, n\} \},
$$
and our goal is to estimate the size of $A_r \cap V^{(k)}$. The expected number of vertices of $V^{(k)}$ that fall into $A_r$ is $|V^{(k)}| / \omega \le \log^2 n+1$ and it follows from Chernoff's bound that with probability at least $1-\exp(-\Theta(\log^2 n))$ it is $O(\log^2 n)$. Using~(\ref{eq:degree}) for the last time, we get that all vertices introduced in this time period have (final) in-degree $O((n/(k \omega \log^2 n) )^{pA_1} \log^2 n)$, provided $k \ge 2$; and $O(n^{pA_1} \log^2 n)$ for $k=1$. It follows that with the desired probability 
\begin{align*}
\sum_{v \in A_r} \deg^-(v) &= O(n^{pA_1} \log^2 n) +  \sum_{k=2}^{k_0} O((n/(k \omega \log^2 n) )^{pA_1} \log^2 n) \cdot O(\log^2 n) \\
&= O(n^{pA_1} \log^2 n) +  O((n/(\omega \log^2 n) )^{pA_1} \log^4 n) \cdot O( k_0^{1-pA_1}) = O(n \log^2 n / \omega),
\end{align*}
and so it holds a.a.s.\ for all $r$. Similarly, using Chernoff's bound and~(\ref{eq:outdegree}) we get that a.a.s.\ for all $r$ we have $|A_r| \sim n/\omega$ and so
$$
\sum_{v \in A_r} \deg^+(v) = O( n / \omega) \cdot O(\log^2 n) = O(n \log^2 n / \omega).
$$
Finally, we are able to get an estimate for the degree tax in $\hat{G_n}$: a.a.s.
\begin{multline*}
\sum_{r=1}^\omega \frac { (\sum_{v \in A_r} \deg(v))^2 }{4|E(G_n)|^2} = \frac {\omega \cdot O( (n \log^2 n / \omega)^2)}{\Theta(n^2)} \\ = O( \log^4 n / \omega)  = O \left( n^{\max \{-1/m,-1+pA_1\} / 2} \log^{9/2} n \right),
\end{multline*}
and the proof is finished.
\end{proof}


\section{Discussion and future research}\label{sec:conclusion}

In this paper, we investigated modularity and provided precise theoretical bounds for several random graph models, such as random $d$-regular graphs, constant average degree graphs, preferential attachment and SPA models. However, there are plenty of directions for future research. 
For example, for preferential attachment model we expect that $q^*(G_m^n) = \Theta(1/\sqrt{m})$.
However, even the fact that $q^*(G_m^n) \to 0$ as $m \to \infty$ is still unproven.

Also, in this paper we studied the most popular version of modularity, while other definitions (suitable for some particular clustering problems) were proposed in the literature (see discussion in~\cite{Fortunato}). For example, it was proposed to multiply the degree tax by a resolution parameter $\gamma$. Note that most of our results can be easily extended to such definition, as we separately estimate edge contribution and degree tax. 
Also, Erd\H{o}s--R{\'e}nyi random graph model can be used as a null model (instead of the pairing model) to compute the degree tax. 
This version of modularity is much easier to analyze, but such null model cannot describe real networks well, since it has an unrealistic Poisson degree distribution.

Finally, we would like to note that there is another model, which, similarly to SPA, combines geometry and preferential attachment~\cite{GeoPA}.  It would be interesting to investigate the modularity for this model and
we expect that its modularity tends to 1 (as for the SPA model). However, these two models are different and our result does not imply anything for the other model.

\section*{Acknowledgements}

This work is supported by Russian Science Foundation (grant number 16-11-10014), NSERC, The Tutte Institute for Mathematics and Computing, and Ryerson University.

\section*{Appendix}

\subsection{Random $d$-regular graphs, some ideas for an upper bound}

\textbf{Idea 1}: Recall that in order to get the current best upper bound we showed that a.a.s.\ no set of size $xn$ induces more than $\bar y(x,d)xn/2$ edges. As a result the largest value of $y_i/d-x_i$ in~(\ref{eq:q_A_sim}) is at most $U_3(d)$. For example, for $d=3$ the optimal choice that maximizes $U_3(3)$ is: $x = \hat{x} \approx 0.0225$, $y = \hat{y} \approx 2.4789$, and so $U_3(3) \approx 0.8038$ as reported in Table~\ref{tab:upper_bound}. However, clearly it is impossible to partition a graph precisely into parts of size $\hat{x} n$. It is possible to show that the following upper bound holds, which is clearly not larger than the previous one:
$$
U_4 = U_4(d) := \max_{k \in \nat \setminus \{1\}} \left( \frac{\bar y(1/k,d)}{d} - 1/k \right).
$$
Unfortunately, this maximum value is achieved for $k=45$ (which corresponds to parts of size roughly $0.0222 n$, and no improvement is achieved: $U_4(3) \approx 0.8038$. The reason this idea fails is that the optimal value of $\hat{x}$ is small so that rounding to the nearest integer for $k$ does not improve the bound much.

\textbf{Idea 2}: Let us look at~(\ref{eq:q_A_sim}) again but this time let us order the terms so that 
$$
\left( \frac{y_1}{d} - x_1 \right) \ge \left( \frac{y_2}{d} - x_2 \right) \ge \ldots \ge \left( \frac{y_k}{d} - x_k \right).
$$
It follows that
$$
q_{\A} = \sum_{i=1}^k x_i \left( \frac{y_i}{d} - x_i \right) \ge x_1 \left( \frac{y_1}{d} - x_1 \right) + (1-x_1) \left( \frac{y_2}{d} - x_2 \right).
$$
It is slightly more tedious than before, but one can get an improvement by considering (ordered) disjoint pairs of vertices $X_1, X_2$ with $|X_1| = x_1n$, $|X_2|=x_2n$, $e(X_1)=y_1x_1n/2$, $e(X_2)=y_2x_2n/2$, and $e(X_1,X_2) = zn$. Unfortunately, this idea also does not provide any reasonable improvement.  For $d=3$, the expected number of pairs of sets for the following vector $(x_1,x_2,y_1,y_2,z) = (\hat{x_1}, \hat{x_2}, \hat{y_1}, \hat{y_2}, \hat{z}) \approx (0.0239,0.0225,2.4830,2.4790,0.000037)$ and, again, no substantial improvement is achieved: $U_5(3) \approx 0.8038$.

\textbf{Idea 3}: As before, let us concentrate on the case $d=3$ but similar ideas can be used for any integer $d \ge 3$. We can try to use the fact that $G_{n,3}$ can be constructed by putting a random matching on the vertices of a Hamiltonian cycle. Let us fix any set of vertices of size $xn$ that induces $zn$ components (paths) by restricting only to edges of the Hamiltonian cycle. Each such set can be represented by the following triple: vertex $v$, vector $(a_1-1, \ldots, a_{zn}-1)$, and vector $(b_1-1, \ldots, b_{zn}-1)$: $v$ starts some path, $a_i$ is the number of vertices on path $i$, $b_i$ is the number of vertices not in the set and right after path $i$. The number of such sets is at most $n {xn \choose zn} {(1-x)n \choose zn}$. The number of edges within this set that are part of the Hamiltonian cycle is $xn-zn$. Hence, in order for the set to induce $yxn/2$ edges, $(yx/2-x+z)n$ edges must be coming from the matching. 

The hope is (that is, \emph{was}) that for small values of $z$, there are only a few sets to consider. On the other hand, if $z$ is closer to $x$, then less edges are ``for free'' (edges of the Hamiltonian cycle). Unfortunately, again this idea does not lead to any substantial improvement. Concentrating on $d=3$, $\hat{x} = 0.0225$, $\hat{y} = 2.4789$, and tuning $\hat{z} \approx 0.00392$, the expected number of such sets is tending to infinity as $n \to \infty$.

\textbf{Conclusion}: The lack of improvement is disappointing but perhaps should not be surprising. Looking at one or two parts of a partition maximizing $q^*$ is not enough (local property). Having one large term $y_i/d-x_i$ in~(\ref{eq:q_A_sim}) might be possible but having all of them to be large perhaps is not. So in order to improve the upper bound, one needs to consider all parts at the same time (global property).

\end{document}